\documentclass[12pt]{amsart}

\usepackage{amsmath,amssymb,amsthm}
\usepackage{cite}
\usepackage[margin=1in]{geometry}
\usepackage{subcaption}
\usepackage[dvipsnames]{xcolor}
\usepackage{longtable}
\usepackage{verbatim} 
\usepackage{ytableau}
\usepackage{booktabs}
\usepackage{tikz}
\usetikzlibrary{positioning}
\usepackage{cleveref}

\DeclareMathOperator{\col}{col}
\DeclareMathOperator{\colpos}{col\,pos}
\DeclareMathOperator{\dynp}{dp}
\DeclareMathOperator{\ind}{ind}
\DeclareMathOperator{\leftof}{left}

\DeclareMathOperator{\row}{row}
\DeclareMathOperator{\std}{std}
\DeclareMathOperator{\successor}{succ}

\DeclareMathOperator{\upof}{up}
\DeclareMathOperator{\pred}{pred}
\newcommand{\met}[1]{\langle #1\rangle}

\DeclareSymbolFont{sfoperators}{OT1}{cmss}{m}{n}
\DeclareSymbolFontAlphabet{\mathsf}{sfoperators}

\makeatletter
\def\operator@font{\mathgroup\symsfoperators}
\makeatother

\renewcommand{\c}[2]{ 
    \ifcase#1
    \or \textcolor{BrickRed}{#2}
    \or \textcolor{ForestGreen}{#2}
    \or \textcolor{RoyalBlue}{#2}
    \or \textcolor{Thistle}{#2}
    \or \textcolor{black}{#2}
    \or \textcolor{gray}{#2}
    \or \textcolor{Tan}{#2}
    \fi
}

\newcommand{\stack}[5]{
    \begin{tikzpicture}[scale=0.6]
        \draw[thick] (-0.5,0) -- (1,0) -- (1,-2) -- (2,-2) -- (2,0) -- (3.5,0);
        \node[fill = white, draw = white] at (2.8,.5) {$#5$};
        \node[fill = white, draw = white] at (1.5,-1.5) {$#4$};
        \node[fill = white, draw = white] at (1.5,-.8) {$#3$};
        \node[fill = white, draw = white] at (1.5,-.1) {$#2$};
        \node[fill = white, draw = white] at (.2,.5) {$#1$};
    \end{tikzpicture}
}
\newcommand{\arrow}{
    \!\!\!\!\!\!\!
    \begin{tikzpicture}[scale=0.6]
        \node at (0,0) {};
        \node at (0,0.7) {$\longrightarrow$};
    \end{tikzpicture}
    \!\!\!\!\!\!\!
}

\theoremstyle{definition}
\newtheorem{theorem}{Theorem}[section]
\newtheorem*{theorem*}{Theorem}
\newtheorem{example}[theorem]{Example}
\newtheorem*{example*}{Example}
\newtheorem{lemma}[theorem]{Lemma}
\newtheorem*{lemma*}{Lemma}
\newtheorem{corollary}[theorem]{Corollary}
\newtheorem*{corollary*}{Corollary}
\newtheorem{definition}[theorem]{Definition}
\newtheorem*{definition*}{Definition}

\newtheorem*{proposition*}{Proposition}
\newtheorem{remark}[theorem]{Remark}
\newtheorem*{remark*}{Remark}
\newtheorem{conjecture}[theorem]{Conjecture}

\title[Stack-sorting Diagrams and Tableaux]{Polynomial Time Enumeration of $t$-Stack-Sortable Permutations Ending in Their Least Entry}
\author{Jerry Zhang}\address{\textsc{J. Zhang}, South Pasadena High School, South Pasadena, CA, 91030} \email{jerrylezhang@gmail.com}
\begin{document}

\begin{abstract}
We study the behavior of West's stack-sorting map $s$ on permutations whose last entry is also their least. Let $S_{n}':=\{\pi0\mid \pi\in S_n\}$ where $\pi0$ denotes the concatenation of $\pi$ and $0$. For each permutation $\pi\in S_n'$, we introduce a new combinatorial object known as the \emph{stack-sorting tableau} $T_{\pi}$, which ultimately serves as the key ingredient in the first polynomial time algorithm for counting the number of $t$-stack-sortable permutations in $S_n'$. We then establish a precise relationship between the behavior of $s$ on $S_{n}'$ and on $S_{n}$. 
\end{abstract}
\maketitle
\section{Introduction}
In 1990, West \cite{west} introduced the stack-sorting map $s$ from the set of permutations to itself. For each permutation $\pi$, the map iteratively pushes entries of $\pi$ into a stack so long as either the stack is empty or the top element of the stack is greater than the next entry of $\pi$. If the next entry cannot be pushed into the stack, then the top element of the stack is popped out of the stack and appended to the output permutation. The process continues until all elements of $\pi$ are appended to the output permutation. Note that throughout the application of $s$ on $\pi$ the stack maintains a strictly increasing order from top to bottom. Further, the maximal unsorted entry of $\pi$ is always sent to its correct position when $s$ is applied to $\pi$, so $s^{n-1}(\pi)$ must be increasing. Let $\met{\pi}$ denote the minimum integer $k$ for which $s^k(\pi)$ is increasing.
\begin{figure}[h]
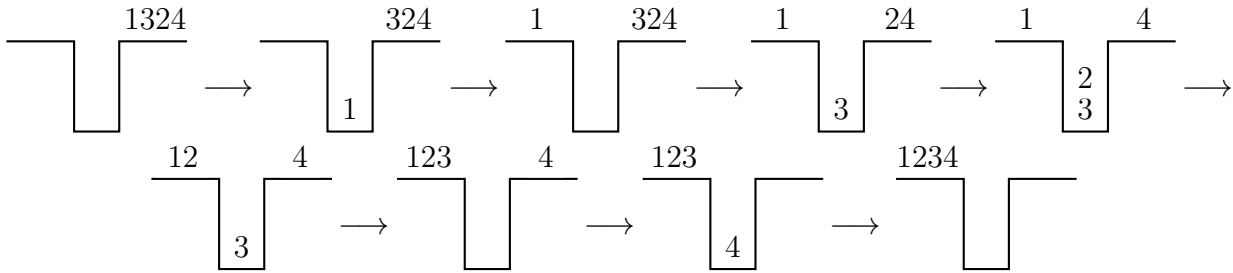

    \begin{center}
        \stack{}{}{}{}{1324}\arrow
        \stack{}{}{}{1}{324}\arrow
        \stack{1}{}{}{}{324}\arrow
        \stack{1}{}{}{3}{24}\arrow
        \stack{1}{}{2}{3}{4}\arrow
        \stack{12}{}{}{3}{4}\arrow
        \stack{123}{}{}{}{4}\arrow
        \stack{123}{}{}{4}{}\arrow
        \stack{1234}{}{}{}{}
    \end{center}
    \caption{West's stack-sorting map applied on $\pi = 1324$.}
    \label{fig:westmap}
\end{figure}

For all $t\in\mathbb{N}$, a permutation $\pi$ is \emph{$t$-stack-sortable} if and only if $\met{\pi}\le t$. The set of all $t$-stack-sortable permutations in $S_n$ is denoted by $\mathcal{W}_t(n)$. In 1968, Knuth showed that $|\mathcal{W}_1(n)|$ was given by the $n$th Catalan number $C_n$.
\begin{theorem}[Knuth \cite{knuth}]
For all $n\in\mathbb{N}$,
\[|\mathcal{W}_1(n)|=C_n=\frac{1}{n+1}\binom{2n}{n}.\]
\end{theorem}
An explicit formula for $|\mathcal{W}_2(n)|$ was conjectured by West and proved by Zeilberger in 1992.
\begin{theorem}[Zeilberger \cite{zeilberger}]
For all $n\in\mathbb{N}$,
\[|\mathcal{W}_2(n)|=\frac{2}{(n+1)(2n+1)}\binom{3n}{n}.\]
\end{theorem}
In 2020, Defant provided the first polynomial time algorithm for computing $|\mathcal{W}_3(n)|$.
\begin{theorem}[Defant \cite{defant}]
For all $n\in\mathbb{N}$, the value $|\mathcal{W}_3(n)|$ is computable in polynomial time.
\end{theorem}
For $t\ge 4$, the enumeration of $|\mathcal{W}_t(n)|$ in polynomial time remains an open problem. Our main result is \Cref{thm:result}, which concerns the enumeration of $t$-stack-sortable permutations ending in their least entry. The set of all $t$-stack-sortable permutations in $S_n'$ is denoted by $\mathcal{W}'_t(n)$.
\begin{theorem}\label{thm:result}
Fix $t\in\mathbb{N}$. For all $n\in\mathbb{N}$, the value $|\mathcal{W}'_t(n)|$ is computable in $O(n^{3t+1})$ time.
\end{theorem}
To justify the relevance of studying stack-sorting on $S_n'$, we then establish a relationship in \Cref{thm:extensiontoSn} between $\met{\pi}$ and $\met{\pi^i}$ for all $\pi\in S_n$ and $i\in[n]$ where $\pi^i$ denotes the permutation obtained by taking the first $i$ entries of $\pi$ and removing any entries less than the $i$th entry of $\pi$. Note that the filtering process leaves $\pi^i$ as a permutation whose last entry is its least.
\begin{theorem}\label{thm:extensiontoSn}
For all $\pi\in S_n$, \[\met{\pi} = \max_{i\in [n]}(\met{\pi^i}).\]
\end{theorem}
\begin{example}
Consider the permutation $\pi=213$. \Cref{thm:extensiontoSn} yields
\[\met{213} = \max(\met{2},\met{21},\met{3}).\]
\end{example}
\section{Preliminaries}
For every permutation $\pi$, we introduce the following notation.
\begin{itemize}
\item The $i$th element of $\pi$ is denoted by $\pi_i$.
\item The size of $\pi$ is denoted by $|\pi|$.
\item For all integers $1\le i<j\le |\pi|+1$, the subpermutation $\pi_i\pi_{i+1}\dots\pi_{j-1}$ is denoted by $\pi[i:j]$. Note that the right endpoint $\pi_j$ is excluded, which mirrors programming slice syntax. We also use the shorthand $\pi[i:]=\pi[i:n+1]$ and $\pi[:j]=\pi[1:j]$.
\item If $i\in\pi$, then the index of $i$ in $\pi$ is denoted by $\ind_{\pi}(i)$.
\item The greatest element of $\pi$ is denoted by $\max(\pi)$.
\item The \emph{standardization} of $\pi$, which is obtained by replacing the $i$th least entry of $\pi$ with $i-1$ for all $1\le i\le |\pi|$, is denoted by $\std(\pi)$. Note that $\std(\pi)\in S_{|\pi|-1}'$.
\item The subsequence of $\pi$ consisting of all elements greater than or equal to $i$ is denoted by $\pi_{\ge i}$.
\end{itemize}
\Cref{thm:extensiontoSn} can now be reformulated under the new notation in \Cref{thm:restated}.
\begin{theorem}[\Cref{thm:extensiontoSn} rephrased]\label{thm:restated}
For all $\pi\in S_n$, \[\met{\pi} = \max_{i\in [n]}(\met{\pi[:i+1]_{\ge\pi_i}}).\]
\end{theorem}
For all $n\in\mathbb{N}$, a \textit{composition} of $n$ is a finite tuple of positive integers whose sum is $n$. The set of all compositions of $n$ is denoted by $\mathcal{C}(n)$. For every composition $\alpha=(\alpha_1,\alpha_2,\dots,\alpha_k)\in\mathcal{C}(n)$, define 
\begin{itemize}
\item $\ell(\alpha):=k$,
\item $w(\alpha):=\max_{i\in[k]} \alpha_i$,
\item $\alpha^{-i}:=(\alpha_1,\dots,\alpha_{i}-1,\dots,\alpha_k)$ for all $i\in [k]$ with $\alpha_{i}>1$,
\item $\alpha^-:=(\alpha_1,\alpha_2,\dots,\alpha_{k-1})$, and
\item $\alpha i:=(\alpha_1,\alpha_2,\dots,\alpha_k,i)$ for all $i\in\mathbb{N}$.
\end{itemize}
A \textit{linear extension} of a poset $(P,\ge^*)$ with $n$ elements is a bijection $f:[n]\to P$ for which $f(x)\ge^* f(y)$ implies $x\ge y$ for all $x,y\in[n]$. The set of all linear extensions of a poset $(P, \ge^*)$ is denoted by $\mathcal{L}(P)$.
\section{Stack-Sorting on $S_n'$}\label{sec:3}
In this section, we examine the evolution of $\pi\in S_n'$ under repeated applications of $s$. We begin by defining \emph{column} and \emph{block} sequences, which arise from identifying the entries of $s^{i-1}(\pi)$ sent to the right of $0$ after a single application of $s$. They serve a very similar purpose to $\mathcal{M}_{\ell}(\pi)$ and $\mathcal{B}_{\ell}(\pi)$, which were previously defined by Defant\cite{defant2} in the context of establishing asymptotic bounds for $\frac{1}{n!}\sum_{\pi\in S_n}\met{\pi}$.

\begin{definition}
For all $\pi\in S_n'$ and $1\le i\le \met{\pi}$, define the recurrent column sequence $C_{\pi,i}=(c_{\pi,i,j})_{j\ge1}$ such that
\[c_{\pi,i,1}:=\max(s^{i-1}(\pi)[:\ind_{s^{i-1}(\pi)}(0)])\]
and
\[c_{\pi,i,j}:=\max(s^{i-1}(\pi)[\ind_{s^{i-1}(\pi)}(c_{\pi,i,j-1})+1:\ind_{s^{i-1}(\pi)}(0)])\]
for all $j\ge 2$ so long as $\ind_{s^{i-1}(\pi)}(c_{\pi, i, j-1})+1<\ind_{s^{i-1}(\pi)}(0)$; otherwise, the recursion terminates.
\end{definition}
\begin{definition}
For all $\pi\in S_n'$ and $1\le i\le \met{\pi}$, define the block sequence $B_{\pi,i}=(b_{\pi,i,j})_{j=1}^{|C_{\pi,i}|}$ such that
\[b_{\pi,i,1}:=s^{i-1}(\pi)[:\ind_{s^{i-1}(\pi)}(c_{\pi,i,1})]\]
and
\[b_{\pi,i,j}:=s^{i-1}(\pi)[\ind_{s^{i-1}(\pi)}(c_{\pi,i,j-1})+1:\ind_{s^{i-1}(\pi)}(c_{\pi,i,j})]\]
for all $2\le j\le |C_{\pi,i}|$.
\end{definition}
\begin{example}\label{ex:ex1}
The running example is the permutation \[\pi = 9\ 3\ 10\ 7\ 8\ 2\ 6\ 1\ 4\ 5\ 0\in S_{10}'.\]
The image of $\pi$ under repeated applications of $s$ is illustrated below with elements of $C_{\pi,i}$ and $B_{\pi,i}$ labelled in red and brown respectively for all $1\le i\le 4$. The elements of $B_{\pi,i}$ are enclosed in parentheses. 
\[
\begin{array}{c}
(\c7{9}\ \c7{3})\ \c1{10}\ (\c7{7})\ \c1{8}\ (\c7{2})\ \c1{6}\ (\c7{1}\ \c7{4})\ \c1{5}\ 0 \\
\downarrow{s} \\
(\c7{3})\ \c1{9}\ \c1{7}\ (\c7{2}\ \c7{1})\ \c1{4}\ 0\ \c6{5}\ \c6{6}\ \c6{8}\ \c6{10} \\
\downarrow{s} \\
\c1{3}\ (\c7{1})\ \c1{2}\ 0\ \c6{4}\ \c6{5}\ \c6{6}\ \c6{7}\ \c6{8}\ \c6{9}\ \c6{10} \\
\downarrow{s} \\
\c1{1}\ 0\ \c6{2}\ \c6{3}\ \c6{4}\ \c6{5}\ \c6{6}\ \c6{7}\ \c6{8}\ \c6{9}\ \c6{10} \\
\downarrow{s} \\
0\ \c6{1}\ \c6{2}\ \c6{3}\ \c6{4}\ \c6{5}\ \c6{6}\ \c6{7}\ \c6{8}\ \c6{9}\ \c6{10}
\end{array}
\]
\end{example}
Immediately, we make three basic observations on the properties of $C_{\pi,i}$ and $B_{\pi,i}$.
\begin{enumerate}
\item $C_{\pi,i}$ is a sequence of right-to-left maxima on the prefix of $s^{i-1}(\pi)$ to the left of $0$.
\item $c_{\pi,i,j}$ is greater than $\max(b_{\pi,i,j})$ for all $1\le j\le |C_{\pi,i}|$.
\item $s^{i-1}(\pi)[:\ind_{s^{i-1}(\pi)}(0)]=b_{\pi,i,1}c_{\pi,i,1}\dots b_{\pi,i,|C_{\pi,i}|}c_{\pi,i,|C_{\pi,i}|}$.
\end{enumerate}
We next express the prefix of $s^i(\pi)$ to the left of $0$ strictly in terms of $B_{\pi,i}$. The following lemma is adapted from Lemma 3.2 of Chen, Luo, and Zhang \cite{words} to the setting of $s$. Similar ideas have also previously been presented by Defant\cite{defant2}.
\begin{lemma}\label{lem:blocks_after_s}
For all $\pi\in S_n'$ and $1\le i\le \met{\pi}$,
\[s^i(\pi)[:\ind_{s^i(\pi)}(0)]=s(b_{\pi,i,1})s(b_{\pi,i,2})\dots s(b_{\pi,i,|C_{\pi,i}|}).\]
\end{lemma}
\begin{proof}
Let $t_j:=j+2(|b_{\pi,i,1}|+|b_{\pi,i,2}|+\dots+|b_{\pi,i,j}|)$ for all $0\le j\le |C_{\pi,i}|$. By induction on $j$, we prove that when $s$ is applied to $s^{i-1}(\pi)$,
\begin{itemize}
\item the remaining input permutation begins with $b_{\pi,i,j+1}\dots c_{\pi,i,|C_{\pi,i}|}0,$
\item the stack consists of $c_{\pi,i,j}c_{\pi,i,j-1}\dots c_{\pi,i,1}$ from top to bottom, and
\item the output permutation is $s(b_{\pi,i,1})s(b_{\pi,i,2})\dots s(b_{\pi,i,j})$
\end{itemize}
after $t_j$ push and pop operations. For $j=0$, the claim is trivial. Now assume that the claim holds for $j-1\ge 0$. After $t_{j-1}$ operations, the remaining input permutation begins with $b_{\pi,i,j}c_{\pi,i,j}$ where $c_{\pi,i,j} > \max(b_{\pi,i,j})$. Further, the top element of the stack becomes $c_{\pi,i,j-1}>\max(b_{\pi,i,j})$ if $j-1>0$, and the stack is empty otherwise. In either case, elements already in the stack can never be popped out, and $c_{\pi,i,j}$ can never be pushed in until after all elements of $b_{\pi,i,j}$ are processed. Thus, in the following $2|b_{\pi,i,j}|+1=t_j-t_{j-1}$ additional operations, $s(b_{\pi,i,j})$ is appended to the output permutation, and then $c_{\pi,i,j}$ is pushed into the stack, which completes the inductive proof. 

In particular, by applying the claim to $j=|C_{\pi,i}|$, we observe that the output permutation is $s(b_{\pi,i,1})s(b_{\pi,i,2})\dots s(b_{\pi,i,|C_{\pi,i}|})$ when $0$ is the first element of the remaining input permutation. Once $0$ is pushed in and immediately popped out by the stack, we finally conclude that
\[s^i(\pi)[:\ind_{s^i(\pi)}(0)]=s(b_{\pi,i,1})s(b_{\pi,i,2})\dots s(b_{\pi,i,|C_{\pi,i}|}).\]
\end{proof}
\begin{corollary}\label{col_well_defined}
For all $\pi\in S_n'$, \[\bigsqcup_{1\le i\le \met{\pi}} \left\{c_{\pi,i,j}:1\le j\le |C_{\pi,i}|\right\}=[n].\]
\end{corollary}
\begin{proof}
By \Cref{lem:blocks_after_s}, $\left\{c_{\pi,i,j}:1\le j\le |C_{\pi,i}|\right\}$ is precisely the set of all elements sent to the right of $0$ when $s$ is applied to $s^{i-1}(\pi)$. Since every element of $[n]$ is eventually sent to the right of $0$, the union of all such sets over $1\le i\le \met{\pi}$ must be $[n]$. Moreover, because elements that are sent to the right of $0$ remain there under subsequent applications of $s$, these sets are pairwise disjoint.
\end{proof}
\begin{corollary}\label{left_well_defined}
For all $\pi\in S_n'$ and $2\le i\le \met{\pi}$, the sequence $C_{\pi,i}$ is a subsequence of $(\max(b_{\pi,i-1,j}))_{j=1}^{|C_{\pi,i-1}|}$.
\end{corollary}
\begin{proof}
By \Cref{lem:blocks_after_s}, both $C_{\pi,i}$ and $(\max(b_{\pi,i-1,j}))_{j=1}^{|C_{\pi,i-1}|}$ occur as subsequences of $s^{i-1}(\pi)$. It therefore suffices to prove that $C_{\pi,i}$ is a subset of $(\max(b_{\pi,i-1,j}))_{j=1}^{|C_{\pi,i-1}|}$. Let $j_k$ be the unique index for which $c_{\pi,i,k}\in b_{\pi,i-1,j_k}$ for all $1\le k\le |C_{\pi,i}|$. If $c_{\pi,i,k}\neq \max(b_{\pi,i-1,j_k})$, then $\max(b_{\pi,i-1,j_k})$ must be between $c_{\pi,i,k}$ and $0$ in $s^{i-1}(\pi)$ by \Cref{lem:blocks_after_s}, which contradicts the definition of $c_{\pi,i,k}$. Hence, $c_{\pi,i,k} = \max(b_{\pi,i-1,j_k})$ for all $1\le k\le |C_{\pi,i}|$, and the result follows.
\end{proof}
Next, we introduce some definitions required for the heart of the argument. In the definitions that follow, the reader may benefit greatly by referring to \Cref{ex:ex2}.
\begin{definition}
For every permutation $\pi\in S_n'$ and integer $i\in[n]$, define $\col_{\pi}(i)$ and $\colpos_{\pi}(i)$ to be the unique integers for which \[c_{\pi,\col_{\pi}(i),\colpos_{\pi}(i)}=i.\] The existence and uniqueness of $\col_{\pi}(i)$ and $\colpos_{\pi}(i)$ are proven in \Cref{col_well_defined}.
\end{definition}
\begin{definition}
For every permutation $\pi\in S_n'$ and integer $i\in[n]$ such that $\col_{\pi}(i)>1$, there exists a unique index $1\le j\le |C_{\pi,\col_{\pi}(i)-1}|$ for which $i=\max(b_{\pi,\col_{\pi}(i)-1,j})$ by \Cref{left_well_defined}. We define \[\leftof_{\pi}(i):=c_{\pi,\col_{\pi}(i)-1,j}.\]
\end{definition}
\begin{definition}
For every permutation $\pi\in S_n'$ and integer $i\in[n]$, define $\row_{\pi}(i)$ recursively such that \[
\row_{\pi}(i):=
\begin{cases}
\colpos_{\pi}(i) & \text{if } \col_{\pi}(i)=1\\
\row_{\pi}(\leftof_{\pi}(i)) & \text{if } \col_{\pi}(i)>1.
\end{cases}
\] Since $\col_{\pi}(\leftof_{\pi}(i))=\col_{\pi}(i)-1$, the recursion must terminate, and $\row_{\pi}(i)$ is well defined.
\end{definition}
\begin{definition}
For every permutation $\pi\in S_n'$ and integer $i\in[n]$ such that $\colpos_{\pi}(i)>1$, define \[\upof_{\pi}(i):=c_{\pi,\col_{\pi}(i),\colpos_{\pi}(i)-1}.\]
\end{definition}
\begin{definition}
For every permutation $\pi\in S_n'$, define the composition 
\[\alpha_{\pi}:=\left(\left|\{i\in[n]: \row_{\pi}(i)=j\}\right|\right)_{j=1}^{|C_{\pi,1}|}.\]
Note that \[\row_{\pi}(i)=\row_{\pi}(\leftof_{\pi}^{\col_{\pi}(i)-1}(i))\in [|C_{\pi,1}|]\] for all $i\in [n]$, so $\alpha_{\pi}$ is a composition of $n$. 
\end{definition}
\begin{definition}
For every composition $\alpha\in \mathcal{C}(n)$, define the \emph{composition diagram}
\[D(\alpha) := \{(i,j)\in \mathbb{N}^2: i\le \alpha_j\}.\]
We impose a partial order relation $\ge_D$ on $D(\alpha)$ such that $(x,y)\ge_{D}(z,w)$ if and only if $x\le z$ and $y\le w$ for all $(x,y),(z,w)\in D(\alpha)$. In particular, $(1,1)$ is the maximal element.
\end{definition}
We refer to $D(\alpha_\pi)$ as the \emph{stack-sorting diagram} of $\pi$.
\begin{definition}
For every permutation $\pi\in S_n'$, define the function $T_{\pi}:[n]\to D(\alpha_\pi)$ such that
\[T_{\pi}(i):=(\col_{\pi}(i), \row_{\pi}(i))\]
for all $i\in[n]$.
\end{definition}
We refer to $T_{\pi}$ as the \emph{stack-sorting tableau} of $\pi$.
\begin{example}\label{ex:ex2}
Recall the example permutation $\pi$ given in \Cref{ex:ex1}. We visualize $T_\pi$ in English notation below. In particular, note that $\leftof_{\pi}(i)$ and $\upof_{\pi}(i)$ precisely refer to the entries immediately left and above of $i$ in the diagram respectively for all $i\in[10]$.
\begin{center}
\begin{minipage}[t]{0.45\textwidth}
\centering
\[
\begin{array}{c}
\c2{9}\ \c3{3}\ \c1{10}\ \c2{7}\ \c1{8}\ \c3{2}\ \c1{6}\ \c4{1}\ \c2{4}\ \c1{5}\ 0 \\
\downarrow{s} \\
\c3{3}\ \c2{9}\ \c2{7}\ \c3{2}\ \c4{1}\ \c2{4}\ 0\ \c1{5}\ \c1{6}\ \c1{8}\ \c1{10} \\
\downarrow{s} \\
\c3{3}\ \c4{1}\ \c3{2}\ 0\ \c2{4}\ \c1{5}\ \c1{6}\ \c2{7}\ \c1{8}\ \c2{9}\ \c1{10} \\
\downarrow{s} \\
\c4{1}\ 0\ \c3{2}\ \c3{3}\ \c2{4}\ \c1{5}\ \c1{6}\ \c2{7}\ \c1{8}\ \c2{9}\ \c1{10} \\
\downarrow{s} \\
0\ \c4{1}\ \c3{2}\ \c3{3}\ \c2{4}\ \c1{5}\ \c1{6}\ \c2{7}\ \c1{8}\ \c2{9}\ \c1{10}
\end{array}
\]
\end{minipage}
\begin{minipage}[t]{0.45\textwidth}
\centering
\[
\ytableaushort{{\c1{10}}{\c2{9}}{\c3{3}},{\c1{8}}{\c2{7}},{\c1{6}},{\c1{5}}{\c2{4}}{\c3{2}}{\c4{1}}}
\]
\end{minipage}
\end{center}
\end{example}
We now work towards showing that $T_{\pi}$ is a linear extension of $D(\alpha_\pi)$.
\begin{lemma}\label{lem:row}
For all $\pi\in S_n'$ and $1\le i\le \met{\pi}$, the sequence $(\row_{\pi}(c_{\pi,i,j}))_{j=1}^{|C_{\pi,i}|}$ is strictly increasing.
\end{lemma}
\begin{proof}
We induct on $i$. For $i=1$, we have \[(\row_{\pi}(c_{\pi,i,j}))_{j=1}^{|C_{\pi,i}|}=(\colpos_{\pi}(c_{\pi,1,j}))_{j=1}^{|C_{\pi,1}|}=(j)_{j=1}^{|C_{\pi,1}|},\] which is strictly increasing. Now assume that the claim holds for $i-1\ge1$. Since
\[(c_{\pi,i,j})_{j=1}^{|C_{\pi,i}|}=(\max(b_{\pi,i-1,\colpos_{\pi}(\leftof_{\pi}(c_{\pi,i,j}))}))_{j=1}^{|C_{\pi,i}|}\]
is a subsequence of $s^{i-1}(\pi)$, the sequence of indices $(\colpos_{\pi}(\leftof_{\pi}(c_{\pi,i,j})))_{j=1}^{|C_{\pi,i}|}$ must be strictly increasing by \Cref{lem:blocks_after_s}, which subsequently implies that
\begin{itemize}
\item $(\leftof_{\pi}(c_{\pi,i,j})))_{j=1}^{|C_{\pi,i}|}$ is a subsequence of $(c_{\pi,i-1,j})_{j=1}^{|C_{\pi,i-1}|}$ and
\item $(\row_{\pi}(\leftof_{\pi}(c_{\pi,i,j})))_{j=1}^{|C_{\pi,i}|}=(\row_{\pi}(c_{\pi,i,j}))_{j=1}^{|C_{\pi,i}|}$ is a subsequence of \\ $(\row_{\pi}(c_{\pi,i-1,j}))_{j=1}^{|C_{\pi,{i-1}}|}$.
\end{itemize} 
By the induction hypothesis, recall that $(\row_{\pi}(c_{\pi,i-1,j}))_{j=1}^{|C_{\pi,{i-1}}|}$ is strictly increasing, and hence $(\row_{\pi}(c_{\pi,i,j}))_{j=1}^{|C_{\pi,i}|}$ is strictly increasing as well.
\end{proof}
\begin{theorem}
Let $\pi\in S_n'$ and $i,j\in [n]$. If $T_{\pi}(i)\ge_{D}T_{\pi}(j)$, then $i\ge j$.
\end{theorem}
\begin{proof}
By basic properties of $C_{\pi,i}$, we can immediately deduce that $\leftof_{\pi}(i)>i$ if $\col_{\pi}(i)>1$ and $\upof_{\pi}(i)>i$ if $\colpos_{\pi}(i)>1$ for all $i\in[n]$. Therefore, for all $i,j\in[n]$ satisfying $T_{\pi}(i)\ge_{D}T_{\pi}(j)$, it suffices to find nonnegative integers $x$ and $y$ such that \[i=\upof_{\pi}^x(\leftof_{\pi}^y(j)).\] 
Consider
\[(x,y):=\left(\colpos_{\pi}\left(\leftof_{\pi}^{\col_{\pi}(j)-\col_{\pi}(i)}(j)\right)-\colpos_{\pi}(i), \col_{\pi}(j)-\col_{\pi}(i)\right).\]
By definition, note that
\[T_{\pi}\left(\leftof_{\pi}^{y}(j)\right)=(\col_{\pi}(i), \row_{\pi}(j)).\]
Since $\row_{\pi}(j)\ge \row_{\pi}(i)$, \Cref{lem:row} then implies that 
\[\colpos_{\pi}\left(\leftof_{\pi}^{y}(j)\right) \ge \colpos_{\pi}(i),\]
and
\[\upof_{\pi}^{x}\left(\leftof_{\pi}^{y}(j)\right) = i,\]
as desired.
\end{proof}
\begin{corollary}
For all $\pi\in S_n'$, the function $T_{\pi}$ is a bijection.
\end{corollary}
\section{Enumerating $t$-Stack-Sortable Permutations in $S_n'$}
In this section, we apply the machinery developed in \Cref{sec:3} to prove \Cref{thm:result}. We begin by reformulating operations previously defined in terms of permutations to be strictly in terms of compositions.  
\begin{definition}
For every composition $\alpha\in\mathcal{C}(n)$ and integer pair $(i,j)\in \mathbb{N}^2$, define
\begin{itemize}
\item $\col_{\alpha}(i,j):=i$,
\item $\row_{\alpha}(i,j):=j$,
\item $\colpos_{\alpha}(i,j):=|U|$,
\item $\leftof_{\alpha}(i,j):=(i-1,j)$, and
\item $\upof_{\alpha}(i,j):=(i,\max\nolimits_{(i,j')\in U}j')$
\end{itemize}
such that $U:=\{(i,j')\in D(\alpha)\mid j'<j\}\cup\{(i,0)\}$.
\end{definition}
\begin{lemma}
Let $\pi\in S_n'$. For all $i\in[n]$,
\begin{itemize}
\item $\col_{\pi}(i)=\col_{\alpha_{\pi}}(T_{\pi}(i))$,
\item $\row_{\pi}(i)=\row_{\alpha_{\pi}}(T_{\pi}(i))$,
\item $\colpos_{\pi}(i)=\colpos_{\alpha_{\pi}}(T_{\pi}(i))$,
\item $T_{\pi}(\leftof_{\pi}(i))=\leftof_{\alpha_{\pi}}(T_{\pi}(i))$ if $\col_{\pi}(i)>1$, and
\item $T_{\pi}(\upof_{\pi}(i))=\upof_{\alpha_{\pi}}(T_{\pi}(i))$ if $\colpos_{\pi}(i)>1$.
\end{itemize}
\end{lemma}
\begin{proof}
The first, second, and fourth statements are immediate by definition. The third and fifth statements are consequences of \Cref{lem:row}.
\end{proof}
\begin{definition}
Let $\alpha\in\mathcal{C}(n)$, $T\in\mathcal{L}(D(\alpha))$, and $\pi\in S_n'$. For every entry $(i,j)\in D(\alpha)$ and integer $k$ satisfying $1\le k<i$, define the \textit{generalized block}
\[B_{\pi,T}(i,j,k):=\begin{cases}s^{k-1}(\pi)\big[\ind_{s^{k-1}(\pi)}(T^{-1}(k,j'))+1 : \ind_{s^{k-1}(\pi)}(T^{-1}(k,j))+1\big] &\text{if }j'>0 \\ s^{k-1}(\pi)\big[: \ind_{s^{k-1}(\pi)}(T^{-1}(k,j))+1\big] &\text{if }j'=0\end{cases}\]
such that $j':=\row_{\alpha}(\upof_{\alpha}(\leftof_{\alpha}(i,j)))$.
\end{definition}
\begin{lemma}\label{lem:generalizedblockstuff}
For all $\pi\in S_n'$, $(i,j)\in D(\alpha_{\pi})$, and $1\le k<i$, 
\[B_{\pi,T_{\pi}}(i,j,k) \subseteq B_{\pi,T_{\pi}}(i,j,1)\]
and
\[B_{\pi,T_{\pi}}(i,j,1) \setminus B_{\pi,T_{\pi}}(i,j,k)\subseteq \{T_{\pi}^{-1}(i',j')\mid (i',j')\in D(\alpha_{\pi}), i'<k\}.\]
\end{lemma}
\begin{proof}
We induct on $k$. For $k=1$, the claim is trivial. Now suppose that the claim holds for $k-1\ge 1$. Let 
\[x := \begin{cases}\colpos_{\alpha_{\pi}}(k-1,j') &\text{if } j'>0 \\ 0 &\text{if } j'=0\end{cases}\]
and
\[y := \colpos_{\alpha_{\pi}}(k-1,j)\]
such that $j':=\row_{\alpha_{\pi}}(\upof_{\alpha_{\pi}}(\leftof_{\alpha_{\pi}}(i,j)))$.

If we express $B_{\pi,T_{\pi}}(i,j,k-1)$ in the form
\[b_{\pi,k-1,x+1}c_{\pi,k-1,x+1}\dots b_{\pi,k-1,y}c_{\pi,k-1,y},\]
then $B_{\pi,T_{\pi}}(i,j,k)$ becomes
\[s(b_{\pi,k-1,x+1})\dots s(b_{\pi,k-1,y})\]
by \Cref{lem:blocks_after_s}, from which we conclude that
\[B_{\pi,T_{\pi}}(i,j,k)\subseteq B_{\pi,T_{\pi}}(i,j,k-1)\]
and
\[B_{\pi,T_{\pi}}(i,j,k-1) \setminus B_{\pi,T_{\pi}}(i,j,k) \subseteq \{T_{\pi}^{-1}(i',j')\mid (i',j')\in D(\alpha_{\pi}), i'=k-1\}.\]
Together with the induction hypothesis, the result follows naturally. 
\end{proof}
Using \Cref{lem:generalizedblockstuff}, we provide an equivalent condition to $T=T_{\pi}$ for all $\alpha\in \mathcal{C}(n)$, $T\in \mathcal{L}(D(\alpha))$, and $\pi\in S_n'$.
\begin{lemma}\label{lem:newcrux}
For all $\alpha\in \mathcal{C}(n)$, $T\in \mathcal{L}(D(\alpha))$, and $\pi\in S_n'$, the equality $T=T_{\pi}$ holds if and only if $(T^{-1}(1,j))_{j=1}^{\ell(\alpha)}$ is a subsequence of $\pi$ and $T^{-1}(i,j)$ lies in $B_{\pi,T}(i,j,1)$ for all $(i,j)\in D(\alpha)$ with $i>1$.
\end{lemma}
\begin{proof}
If $T=T_{\pi}$, then $(T^{-1}(1,j))_{j=1}^{\ell(\alpha)}$ is equal to $C_{\pi, 1}$, which is a subsequence of $\pi$. Moreover, by \Cref{lem:generalizedblockstuff},
\[T_{\pi}^{-1}(i,j)\in B_{\pi,T_{\pi}}(i,j,i-1)\subseteq B_{\pi,T_{\pi}}(i,j,1).\]
for all $(i,j)\in D(\alpha)$ satisfying $i>1$.

Conversely, if $(T^{-1}(1,j))_{j=1}^{\ell(\alpha)}$ is a subsequence of $\pi$ and $T^{-1}(i,j)\in B_{\pi,T}(i,j,1)$ for all $(i,j)\in D(\alpha)$ with $i>1$, we prove that $T^{-1}(i,j)=T^{-1}_{\pi}(i,j)$ by strong induction on $(i,j)$ in the poset $(D(\alpha),\ge_{D})$. For the maximal element $(i,j)=(1,1)$, we have 
\[T^{-1}(1,1)=n=T_{\pi}^{-1}(1,1).\]
Now assume that the claim holds for all $(k,\ell)\in D(\alpha)$ such that either $k<i$ or $k=i$ and $\ell < j$.  If $i=1$, then every element of $\{T^{-1}(i',j')\mid (i',j')\in D(\alpha), j'<j\}$ is to the left of $T^{-1}(1,j-1)$ in $\pi$. Therefore, $T^{-1}(1,j)$ is the maximal element between $T^{-1}(1,j-1)=c_{\pi,1,j-1}$ and $0$ in $\pi$, which is precisely the definition of $c_{\pi,1,j}=T_{\pi}^{-1}(1,j)$. If $i>1$, then first note that
\[T^{-1}(i,j) \in B_{\pi,T_{\pi}}(i,j,i-1)\]
by the induction hypothesis and \Cref{lem:generalizedblockstuff} since
\[T^{-1}(i,j)\in B_{\pi,T}(i,j,1) = B_{\pi,T_{\pi}}(i,j,1)\]
and
\begin{align*}
B_{\pi,T_{\pi}}(i,j,1) \setminus B_{\pi,T_{\pi}}(i,j,i-1)&\subseteq \{T_{\pi}^{-1}(i',j')\mid (i',j')\in D(\alpha_{\pi}), i'<i-1\} \\
&\subseteq \{T^{-1}(i',j')\mid (i',j')\in D(\alpha), i'<i-1\}.
\end{align*}
Further, $T^{-1}(i,j)$ is the greatest element of $B_{\pi,T_{\pi}}(i,j,1)$ besides $T^{-1}(i-1,j)$ because no element of $\{T^{-1}(i',j')\mid (i',j')\in D(\alpha), i'\ge i-1,j'<j\}$ is in $B_{\pi,T}(i,j,1)$ and, consequently, $B_{\pi,T}(i,j,i-1)$. This finally implies that $T^{-1}(i,j)=\max(b_{\pi,i-1,\colpos(T^{-1}_{\pi}(i-1,j))})$ and \[\leftof_{\pi}(T^{-1}(i,j))=T^{-1}_{\pi}(i-1,j)=\leftof_{\pi}(T^{-1}_{\pi}(i,j)).\]
\end{proof}
We now introduce the hook length function for stack-sorting tableaux. 
\begin{definition}
For every composition $\alpha\in\mathcal{C}(n)$, define the \emph{hook length} function $h_{\alpha}:\mathbb{N}^2\to \mathbb{N}$ such that
\[h_{\alpha}(i,j):=\left|\left\{(i',j')\in D(\alpha)\mid i'\le i-1,\row_{\alpha}(\upof_{\alpha}(i-1,j))<j'\le j\right\}\right|\]
if $i>1$ and
\[h_{\alpha}(i,j):=1\]
otherwise for all $(i,j)\in \mathbb{N}^2$.
\end{definition}
\begin{remark}\label{rem:hook}
Note that we may alternatively express
\[h_{\alpha}(i,j)=\begin{cases}\sum_{j'=\row_{\alpha}(\upof_{\alpha}(i-1,j))+1}^{j-1} \alpha_{j'} + \min(i-1,\alpha_j) & \text{if } i>1 \\ 1 &\text{if }i=1\end{cases}\]
because every row between $\row_{\alpha}(\upof_{\alpha}(i-1,j))+1$ and $j-1$ in $\alpha$ has length less than $i-1$ for $i>1$.
\end{remark}
\begin{example}
Consider the composition $\alpha=(3,2,1,4)$. We label each cell $(i,j)\in D(\alpha)$ with $h_{\alpha}(i,j)$.
\[\ytableaushort{{1}{1}{2},{1}{1},{1},{1}{1}{3}{6}}\]
\end{example}
\begin{definition}
For every subset $S\subseteq \{(\alpha,T)\mid \alpha\in \mathcal{C}(n),T\in \mathcal{L}(D(\alpha))\}$, let $\mathcal{P}(S)$ denote the set of all permutations $\pi\in S_n'$ that satisfy $(\alpha_{\pi},T_{\pi})\in S$.
\end{definition}
\begin{theorem}\label{thm:PT}
For all $\alpha\in\mathcal{C}(n)$ and $T\in\mathcal{L}(D(\alpha))$,
\[|\mathcal{P}(\{(\alpha,T)\})| = \prod_{(i,j)\in D(\alpha)}h_{\alpha}(i,j).\]
\end{theorem}
\begin{proof}
We induct on $n$. If $w(\alpha)=1$, then $\pi$ must be strictly decreasing by \Cref{lem:newcrux}, so there is only one permutation in $\mathcal{P}(\{(\alpha,T)\})$, which agrees with the calculation
\[\prod_{(i,j)\in D(\alpha)}h_{\alpha}(i,j)=\prod_{(i,j)\in D(\alpha)}1=1.\] Now assume that $w(\alpha)>1$ and the claim holds for $n-1\ge 1$. In particular, consider the linear extension $T|_{D(\alpha)\setminus \{(i,j)\}}$ where $(i,j)$ is a minimal element of $D(\alpha)$ and $i>1$. For all $\pi\in S_n'$, observe that $\pi\in\mathcal{P}(\{\alpha,T\})$ if and only if $T^{-1}(i,j)\in B_{\pi,T}(i,j,1)$ and 
\[\pi\setminus \{T^{-1}(i,j)\}\in \mathcal{P}(\{(\alpha^{-j},T|_{D(\alpha)\setminus (i,j)})\})\]
by \Cref{lem:newcrux}. Therefore, we may obtain the set $\mathcal{P}(\{\alpha, T\})$ by inserting the element $T^{-1}(i,j)$ into $B_{\pi,T}(i,j,1)$ for each permutation $\pi'\in\mathcal{P}(\{(\alpha^{-j},T|_{D(\alpha)\setminus (i,j)})\})$. There are always
\begin{align*}
|B_{\pi,T}(i,j,1)| &=|\{(i',j')\in D(\alpha_{\pi})\mid i'\le i-1,\row_{\alpha}(\upof_{\alpha}(i-1,j))<j'\le j\}| \\
&=h_{\alpha}(i,j)
\end{align*}
possible positions to insert $T^{-1}(i,j)$, so we conclude that
\begin{align*}
|\mathcal{P}(\{(\alpha,T)\})| &= h_{\alpha}(i,j)\cdot |\mathcal{P}(\{(\alpha^{-j},T|_{D(\alpha)\setminus (i,j)})\})| \\
&=\prod_{(i,j)\in D(\alpha)}h_{\alpha}(i,j).
\end{align*}
\end{proof}
\begin{corollary}
For all $\alpha\in\mathcal{C}(n)$,
\[|\mathcal{P}(\{(\alpha,T):T\in\mathcal{L}(D(\alpha))\})|=|\mathcal{L}(D(\alpha))|\cdot\prod_{(i,j)\in D(\alpha)}h_{\alpha}(i,j).\]
\end{corollary}
\begin{definition}
Let $m,w,w_{\ell}\in\mathbb{N}\cup\{0\}$. For all sequences $h=(h_i)_{i=1}^w$ and $p=(p_i)_{i=1}^w$, define
\[S(m,w,w_{\ell},h,p):=\left\{(\alpha, T)\middle| \begin{aligned}&\alpha\in\mathcal{C}(m),T\in\mathcal{L}(D(\alpha)),w=w(\alpha),w_{\ell}=\alpha_{\ell(\alpha)},\\&h=(h_{\alpha}(i+1,\ell(\alpha)))_{i=1}^w,\\& p=(T^{-1}(\upof_{\alpha}(i,\ell(\alpha)+1)))_{i=1}^{w}\end{aligned}\right\}.\]
\end{definition}
Here, $h$ tracks the hook lengths of the cells on the last row, while $p$ tracks the minimal entry for each column.
\begin{definition}
Let $n,t\in \mathbb{N}$. Define the set of states
\[\mathbb{S}(n,t):=\left\{(m,w,w_{\ell},h,p)\mid m\le n, w_{\ell}\le w\le t, h\in[m]^{w},p\in[m]^{w}\right\}.\]
\end{definition}
\begin{definition}
For every state $(m,w,w_{\ell},h,p)\in\mathbb{S}(n,t)$ and integer $i\in [t]$, 
\begin{itemize}
\item let $w^i:=\max(w,i)$,
\item let $h^{i}=(h_j^i)_{j=1}^{w^i}$ such that
\[h^i_j:=\begin{cases}\min(j,i) & \text{if } 1\le j\le w_{\ell}\\ h_{j}+\min(j,i)&\text{if } w_{\ell}<j\le w \\ m+j&\text{if } w<j\le w^i\end{cases}\]
for all $j\in[w^i]$, and
\item let $P^i$ be the set of all $(p_j^i)_{j=1}^{w^i}\in [m+i]^{w^i}$ satisfying \begin{itemize}\item $p_1^i > p_2^i>\dots>p_i^i$,\item $p_j^i<p_j+|\{k\in [i]: p_k^i\le p_j+i-k\}|$ for all $j\le \min(w,i)$, \item $p_j^i=p_j+|\{k\in [i]: p_k^i\le p_j+i-k\}|$ for all $i<j\le w$. \end{itemize}
\end{itemize}
We define the set of successors and predecessors of $(m,w,w_{\ell},h,p)$ by
\[\successor(m,w,w_{\ell},h,p) := \left\{(m+i,w^i,i,h^i,p^i)\mid i\in[t],p^i\in P^i\right\}\]
and
\[\pred(m,w,w_{\ell},h,p) := \{x\in \mathbb{S}(n,t)\mid (m,w,w_{\ell},h,p)\in \successor(x)\}\]
respectively.
\end{definition}
\begin{example}
For the state \[x=(6,3,1,(1,1,3),(2,3,1))\in \mathbb{S}(100,4)\] and $i=4$, we have 
\begin{itemize}
\item $w^i=4$, 
\item $h^i=(1,3,6,10)$, and 
\item $(5,4,2,1)\in P^i$.
\end{itemize}
\end{example}
\begin{remark}
For all $(\alpha,T)\in S(m,w,w_{\ell},h,p)$, note that the set $P^i$ characterizes all sequences $(p_{j}^i)_{j=1}^{w^i}$ for which inserting $(j,\ell(\alpha)+1)$ into position $p_j^i$ of $T$ for all $j\in [i]$ yields a new linear extension \[T^i:[n+i]\to D(\alpha i)\]
such that the updated position of $(j,\ell(\alpha))$ in $T^i$ is $p_j^i$ for all $i< j\le w$. Further, by \Cref{rem:hook}, note that \[h^i=(h_{\alpha i}(i+1,\ell(\alpha i)))_{i=1}^{w^i}.\]
Therefore, we observe that the set $\cup_{x\in \pred(n,w,w_{\ell},h,p)} S(x)$ consists of all linear extensions derived by removing the last row from each $(T,\alpha)\in S(m,w,w_{\ell},h,p)$.
\end{remark}
\begin{lemma}\label{lem:recursion}
For all $(m,w,w_{\ell},h,p)\in \mathbb{S}(n,t)$ such that $m\ge 1$,
\[|\mathcal{P}(S(m,w,w_{\ell},h,p))|=\sum_{x\in \pred(m,w,w_{\ell},h,p)}|\mathcal{P}(S(x))|\cdot \prod_{i=1}^{w_{\ell}-1}h_i.\]
\end{lemma}
\begin{proof}
Fix $(\alpha,T)\in S(m,w,w_{\ell},h,p)$. By \Cref{thm:PT},
\begin{align*}
|\mathcal{P}(\{(\alpha,T)\})| &= \prod_{(i,j)\in D(\alpha)}h_\alpha(i,j) \\
&= \prod_{(i,j)\in D(\alpha^-)}h_{\alpha^-}(i,j)\cdot \prod_{i=1}^{w_{\ell}}h_{\alpha}(i,\ell(\alpha)) \\
&= \prod_{(i,j)\in D(\alpha^-)}h_{\alpha^-}(i,j)\cdot \prod_{i=1}^{w_{\ell}-1}h_i \\
&= |\mathcal{P}(\{(\alpha^-,T|_{D(\alpha^-)})\})|\cdot \prod_{i=1}^{w_{\ell}-1}h_i.
\end{align*}
Summing over all $(\alpha,T)\in S(m,w,w_{\ell},h,p)$, we obtain the desired result.
\end{proof}
\begin{proof}[Proof of \Cref{thm:result}]
We first present a forward dynamic programming algorithm that computes $|\mathcal{P}(S(x))|$ for all $x\in \mathbb{S}(n,t)$. First, initialize $\dynp(x)=1$ if $x=(0,0,0,(),())$ and $\dynp(x)=0$ otherwise for all $x\in \mathbb{S}(n,t)$. Then, iterating over all $m\in \{0,1,\dots,n-1\}$, we update 
\[\dynp(x')\leftarrow \dynp(x') + \dynp(x) \cdot \prod_{j=1}^{i-1} h_j^i\]
for all $x=(m,w,w_{\ell},h,p)\in\mathbb{S}(n,t)$ and $x'=(m+i,w^i,i,h^i,p^i)\in \successor(x)$. When the algorithm concludes, \[\dynp(x)=|\mathcal{P}(S(x))|\] by \Cref{lem:recursion}. Each update requires $O(1)$ time. There are \[|\mathbb{S}(n,t)|=O(n\cdot t\cdot t\cdot n^t\cdot n^t)=O(n^{2t+1})\] total states, and for each state $x\in \mathbb{S}(n,t)$, the set of successors $\successor(x)$ is computable in $O(n^t)$ time. Therefore, the overall time complexity is $O(n^{3t+1})$. Then, to compute $|\mathcal{W}_t'(n)|$, it suffices to simply sum
\[|\mathcal{W}_t'(n)|=\sum_{(n,w,w_{\ell},h,p)\in\mathbb{S}(n,t)} |\mathcal{P}(S(n,w,w_{\ell},h,p))|\]
since a permutation $\pi\in S_n'$ is $t$-stack-sortable if and only if $w(\alpha_{\pi})\le t$.
\end{proof}
\section{Stack-Sorting on $S_n$}
We now shift our attention to proving \Cref{thm:extensiontoSn}. The following lemma is partially adapted from Lemma 3.1 of Chen, Luo, and Zhang \cite{words} to the setting of $s$.
\begin{lemma}\label{lem:increasingtail}
Let $\pi\in S_n$ and $k\ge 0$. The permutation $s^k(\pi)[\ind_{s^k(\pi)}(\pi_n)+1:]$ is sorted, and its elements are all greater than $\pi_n$.
\end{lemma}
\begin{proof}
We induct on $k$. For $k= 0$, the permutation $\pi[\ind_{\pi}(\pi_n)+1:]$ is empty, so the claim is vacuously true. Now assume that the claim holds for $k-1\ge 0$. Consider the moment immediately after $\pi_n$ is popped from the stack when $s$ is applied to $s^{k-1}(\pi)$. Let the remainder of the input permutation be $r = r_1r_2\dots r_i$, and let the current elements of the stack from top to bottom be $t=t_1t_2\dots t_j$. By the induction hypothesis, $r$ is sorted, and its elements are all greater than $\pi_n$. Furthermore, because the stack always maintains a strictly increasing order and because $\pi_n$ was just popped out of the stack, $t$ must be sorted, and its elements are also all greater than $\pi_n$.

Now, if $t_1 < r_1$, then $t_1$ is popped out. If $t_1 > r_1$, then $r_1$ is pushed into the stack and popped out immediately. In either case, the least element is always appended to the output permutation. Thus, after $i+j$ iterations, the subpermutation of $s^k(\pi)$ to the right of $\pi_n$ will be sorted, and its elements are all greater than $\pi_n$.
\end{proof}
\begin{lemma}\label{lem:excludinglast}
Let $\pi\in S_n$ and $k\ge 0$. The permutation $s^k(\pi[:n])$ occurs as a subpermutation of $s^k(\pi)$.
\end{lemma}
\begin{proof}
We induct on $k$. For $k=0$, the claim trivially holds. Now assume that the claim is true for $k-1\ge0$. Then, we consider the effect of the addition of $\pi_n$ to $s^{k-1}(\pi[:n])$ at position $\ind_{s^{k-1}(\pi)}(\pi_n)$ when $s$ is applied. Recall that by \Cref{lem:increasingtail} all elements of $s^{k-1}(\pi)$ to the right of $\pi_n$ are greater than $\pi_n$. Therefore, whenever an element is popped out of the stack by $\pi_n$ while applying $s$ to $s^{k-1}(\pi)$, it would also be popped out by the immediately following element or by exhaustion of the input permutation while applying $s$ to $s^{k-1}(\pi[:n])$. Thus, we conclude that the ordering of $s^{k-1}({\pi[:n]})$ upon applying $s$ is never disrupted by adding $\pi_n$, and $s^k(\pi)$ always includes the subpermutation $s^k(\pi[:n])$.
\end{proof}
We also cite Lemma 2.1 from Defant\cite{defant2}, and we include a proof for completeness.
\begin{lemma}\label{lem:gelastk}
For all $\pi\in S_n$, $m\in [n]$, and $k\ge 0$,
\[s^{k}(\pi_{\ge m})=s^k(\pi)_{\ge m}.\]
\end{lemma}
\begin{proof}
We first use strong induction on $n$ to prove that $s(\pi_{\ge m}) = s(\pi)_{\ge m}$. For $n=1$, the claim is trivial. Now assuming that the claim holds for all $n\le \ell - 1$ such that $\ell\ge 2$, we prove that the claim holds for $\pi\in S_\ell$. If we express $\pi$ as $LnR$, then $s(\pi)=s(L)s(R)n$ by the recursive definition of the stack-sorting map. Similarly, if we express $\pi_{\ge m}$ as $L_{\ge m}nR_{\ge m}$, then $s(\pi_{\ge m})=s(L_{\ge m})s(R_{\ge m})n.$ Notably, $L_{\ge m}$ and $R_{\ge m}$ are subpermutations of $L$ and $R$ respectively, so under the induction hypothesis, $s(L_{\ge m})$ and $s(R_{\ge m})$ occur as subpermutations of $s(L)$ and $s(R)$ respectively as $|L|, |R|\le \ell-1$. Therefore, we conclude that $s(\pi_{\ge m})=s(L_{\ge m})s(R_{\ge m})n$ occurs as a subpermutation of $s(\pi)=s(L)s(R)n$, and $s(\pi_{\ge m}) = s(\pi)_{\ge m}$.

Next, we induct on $k$ to prove the general claim. For $k=0$, the claim is trivial. Suppose that the claim holds for $k-1$. Then, note that $s^{k-1}(\pi_{\ge m}) = s^{k-1}(\pi)_{\ge m}$, and \[s^k(\pi_{\ge m})=s(s^{k-1}(\pi_{\ge m})) = s(s^{k-1}(\pi)_{\ge m}) = s^k(\pi)_{\ge m},\]
which completes the inductive proof.
\end{proof}
\begin{lemma}\label{lem:reduction}
For all $\pi\in S_n$, it holds that $\met{\pi}=\max(\met{\pi[:n]},\met{\pi_{\ge \pi_n}})$.
\end{lemma}
\begin{proof}
By \Cref{lem:excludinglast} and \Cref{lem:gelastk}, it is evident that \[\met{\pi}\ge\max(\met{\pi[:n]},\met{\pi_{\ge \pi_n}})\] because $s^k(\pi[:n])$ and $s^k(\pi_{\ge \pi_n})$ are both subpermutations of $s^k(\pi)$ for all $k\ge 0$. If $\met{\pi} > \met{\pi[:n]}$, then $\pi_n$ must be the only element out of order in $s^{\met{\pi[:n]}}(\pi)$. Note that $\pi_n$ is always to the right of elements less than $\pi_n$, so when the subpermutation $s^{\met{\pi_{\ge\pi_n}}}(\pi_{\ge \pi_n})$ is sorted, $s^{\met{\pi_{\ge\pi_n}}}(\pi)$ must then also be sorted.
\end{proof}
\begin{proof}[Proof of \Cref{thm:extensiontoSn}]
We induct on $n$. For $n=1$, the claim is trivial because $\pi[:2]_{\ge \pi_1}=\pi$. Now assume that the claim holds for $n-1\ge 1$. By \Cref{lem:reduction},
\begin{align*}
\met{\pi}&=\max(\met{\pi[:n]},\met{\pi_{\ge \pi_n}}) \\
&=\max\left(\max_{i\in [n-1]}(\met{\pi[:n][:i+1]_{\ge \pi[:n]_i}}),\met{\pi_{\ge \pi_n}}\right) \\
&= \max_{i\in [n]}(\met{\pi[:i+1]_{\ge \pi_i}}).
\end{align*}
\end{proof}
\section{Future Directions}
We conclude by providing some insight into how one might generalize our approach to $S_n$ or other special subsets of $S_n$. Note that \Cref{thm:extensiontoSn} implies
\[\met{\pi}=\max_{i\in [n]}\left(w\left(\alpha_{\std\left(\pi[:i+1]_{\ge\pi_i}\right)}\right)\right).\]
Thus, exploring the evolution of the sequence $\left(T_{\std\left(\pi[:i+1]_{\ge\pi_i}\right)}\right)_{i=1}^n$ may be very promising for enumerating $W_t(n)$. However, there is no natural generalization of stack-sorting diagrams or tableaux for $S_n$, which prohibits the same dynamic programming argument from succeeding. $S_n'$ admits a polynomial time enumeration of its $t$-stack-sortable elements with our method precisely due to the fact that $s^t(\pi) \in S_n'$ is sorted if and only if $s^t(\pi_{\ge {\pi_n}})$ is sorted, so it limits our exploration to only considering elements sent to the right of $0$ after every application of $s$. Incidentally, if we let $S_n^2:=\{\pi\in S_{n+1}\mid \pi_{n+1}=2\}$, then the same property would hold for $S_n^2$ because $2$ is always to the right of $1$ in $s^t(\pi)$ for all $\pi\in S_n^2$ and $t\ge 0$. Therefore, $\met{\pi}=\met{\std(\pi_{\ge2})}$ for all $\pi\in S_n^2$, and the number of $t$-stack-sortable permutations in $S_n^2$ is exactly $|\mathcal{W}'_t(n-1)|$, which is computable in polynomial time.

Additionally, in \Cref{app}, we observe that $|\mathcal{W}'_2(n)|$ corresponds to the Motzkin numbers (OEIS A001006 \cite{oeis_A001006}), so we accordingly make \Cref{conj}. There is yet to be a closed form or simple recursive formulation of $|\mathcal{W}'_t(n)|$ for $t\ge 3$.
 
\begin{conjecture}\label{conj}
For all $n\in \mathbb{N}$, the value $|W_2'(n)|$ is the $n$th Motzkin number.
\end{conjecture}
Another future direction worth exploring is improving the time complexity bound for computing $|W_t'(n)|$. With prefix sums, we can show that $|W_t'(n)|$ is computable in $O(n^{2t+1})$ time, though we choose to exclude the proof because achieving polynomial time complexity is fundamentally the purpose of this paper.
\section{Appendix} \label{app}
We provide the computed values of $|\mathcal{W}'_t(n)|$ for $t\in\{2,3,4\}$ and $n\in[30]$.
\begin{longtable}{rr|rr}
\toprule
$n$ & $|\mathcal{W}'_2(n)|$ & $n$ & $|\mathcal{W}'_2(n)|$ \\ \midrule
1 & 1 & 16 & 853,467 \\
2 & 2 & 17 & 2,356,779 \\
3 & 4 & 18 & 6,536,382 \\
4 & 9 & 19 & 18,199,284 \\
5 & 21 & 20 & 50,852,019 \\
6 & 51 & 21 & 142,547,559 \\
7 & 127 & 22 & 400,763,223 \\
8 & 323 & 23 & 1,129,760,415 \\
9 & 835 & 24 & 3,192,727,797 \\
10 & 2,188 & 25 & 9,043,402,501 \\
11 & 5,798 & 26 & 25,669,818,476 \\
12 & 15,511 & 27 & 73,007,772,802 \\
13 & 41,835 & 28 & 208,023,278,209 \\
14 & 113,634 & 29 & 593,742,784,829 \\
15 & 310,572 & 30 & 1,697,385,471,211 \\ \bottomrule
\end{longtable}
\begin{longtable}{rr|rr}
\toprule
$n$ & $|\mathcal{W}'_3(n)|$ & $n$ & $|\mathcal{W}'_3(n)|$ \\ \midrule
1 & 1 & 16 & 460,125,335 \\
2 & 2 & 17 & 2,165,580,695 \\
3 & 6 & 18 & 10,291,948,854 \\
4 & 18 & 19 & 49,345,393,406 \\
5 & 60 & 20 & 238,493,417,444 \\
6 & 218 & 21 & 1,161,146,210,522 \\
7 & 826 & 22 & 5,691,351,451,536 \\
8 & 3,261 & 23 & 28,069,230,225,236 \\
9 & 13,337 & 24 & 139,228,254,682,547 \\
10 & 56,056 & 25 & 694,262,710,142,607 \\
11 & 241,206 & 26 & 3,479,021,348,150,096 \\
12 & 1,059,255 & 27 & 17,513,828,374,589,112 \\
13 & 4,733,887 & 28 & 88,545,050,076,393,080 \\
14 & 21,483,097 & 29 & 449,456,637,011,361,626 \\
15 & 98,825,193 & 30 & 2,290,043,872,282,754,031 \\ \bottomrule
\end{longtable}
\begin{longtable}{rr|rr}
\toprule
$n$ & $|\mathcal{W}'_4(n)|$ & $n$ & $|\mathcal{W}'_4(n)|$ \\ \midrule
1 & 1 & 16 & 15,921,264,079 \\
2 & 2 & 17 & 104,757,767,491 \\
3 & 6 & 18 & 699,855,355,916 \\
4 & 24 & 19 & 4,740,919,917,872 \\
5 & 96 & 20 & 32,526,908,642,094 \\
6 & 420 & 21 & 225,785,182,393,596 \\
7 & 2,004 & 22 & 1,584,240,656,499,096 \\
8 & 10,248 & 23 & 11,227,118,610,129,040 \\
9 & 54,558 & 24 & 80,301,687,416,204,615 \\
10 & 301,964 & 25 & 579,308,930,683,932,451 \\
11 & 1,732,408 & 26 & 4,212,822,807,107,915,984 \\
12 & 10,256,360 & 27 & 30,866,455,025,733,336,786 \\
13 & 62,322,928 & 28 & 227,743,434,563,963,874,771 \\
14 & 387,557,130 & 29 & 1,691,470,337,203,992,553,815 \\
15 & 2,460,804,208 & 30 & 12,640,748,036,006,674,578,379 \\ \bottomrule
\end{longtable}
\section*{Conflict of Interest Statement}

On behalf of all authors, the corresponding author states that there is no conflict of interest.

\section*{Data Availability Statement}

We do not analyze any datasets.
\bibliographystyle{plain}
\bibliography{bib_stacksortingdiagrams}
\end{document}